\documentclass[12pt]{amsart}

\usepackage{amsmath}
\usepackage{amssymb}
\usepackage{amsthm}
\usepackage{xcolor}
\usepackage{bm}
\usepackage{graphicx}
\usepackage{pifont}
\usepackage{mathrsfs}
\usepackage{epsfig,verbatim}

\newcommand\be{\begin{equation}}
\newcommand\ee{\end{equation}}
\newcommand\bea{\begin{eqnarray}}
\newcommand\eea{\end{eqnarray}}
\newcommand\beaa{\begin{eqnarray*}}
\newcommand\eeaa{\end{eqnarray*}}
\newcommand\beba{\begin{equation}\left\{\begin{array}{rcl}}
\newcommand\eeba{\end{array}\right.\end{equation}}
\newcommand\bebaa{\begin{equation*}\left\{\begin{array}{rcl}}
\newcommand\eebaa{\end{array}\right.\end{equation*}}

\newcommand\bR{{\mathbb{R}}}
\newcommand\bC{{\mathbb{C}}}
\newcommand\bN{{\mathbb{N}}}

\newcommand\cZ{{\mathcal{Z}}}
\newcommand\cN{{\mathcal{N}}}

\newcommand{\dint}{\displaystyle \int}
\newcommand{\dsum}{\displaystyle \sum}

\newcommand{\dlim}{\displaystyle \lim}

\newcommand{\lng}{\left\langle  \right.}
\newcommand{\rng}{\left.  \right\rangle}
\newtheorem{theorem}{Theorem}[section]
\newtheorem{lemma}[theorem]{Lemma}
\newtheorem{corollary}[theorem]{Corollary}

\newtheorem{remark}[theorem]{Remark}
\newtheorem{proposition}[theorem]{Proposition}

\numberwithin{equation}{section}

\begin{document}

\title{Asymptotic profiles of zero points of solutions to the heat equation}

\author[H. Ishii]{Hiroshi Ishii}
\address{Department of Mathematics, Hokkaido University, Sapporo, 060-0810, Japan}
\email{ishii.hiroshi.7n@kyoto-u.ac.jp}

\thanks{Date: \today. Corresponding author: H. Ishii}


\begin{abstract}
In this paper, we consider the asymptotic profiles of zero points for the spatial variable of the solutions to the heat equation.
By giving suitable conditions for the initial data, 
we prove the existence of zero points by extending the high-order asymptotic expansion theory for the heat equation.
This reveals a previously unknown asymptotic profile of zero points diverging at $O(t)$.
In a one-dimensional spatial case, we show the zero point's second and third-order asymptotic profiles in a general situation.
We also analyze a zero-level set in high-dimensional spaces and obtain results that extend the results for the one-dimensional spatial case.
\end{abstract}

\maketitle
\section{Introduction}
We consider a Cauchy problem 
\be\label{eq:main}
\begin{cases}
\dfrac{\partial u}{\partial t}=\Delta_{x} u \quad (t>0,\ x\in\bR^{d}), \vspace{3mm}\\
u(0,x)=u_0(x) \quad (x\in\bR^d),
\end{cases}
\ee
where $u=u(t,x)\in\bR\ (t>0,x\in\bR^{d})$, $\Delta_{x}$ is Laplace operator for the spatial variable $x$, and $u_0\in L^1(\bR^{d})\cap L^{\infty}(\bR^{d})$.
Throughout this paper, we assume that $\| u_0 \|_{L^1}\neq 0$ and denote the Euclidean norm and inner product on $\bR^{d}$ by $|\cdot|$ and $\lng\cdot,\cdot\rng$, respectively.

In this paper, we consider the asymptotic behavior of the zero level set $\cZ(t):= \{ x\in\bR^d \mid u(t,x)=0 \}$,
where $u(t,{x})$ is a unique bounded solution of \eqref{eq:main} which can be written as follows:
\be\label{basic sol.}
u(t,{x})=(G(t)*u_0)(x)=\int_{\bR^{d}} G(t,x-y) u_0(y)dy\quad (t>0,\ x\in\bR^{d})
\ee
of \eqref{eq:main}.
Here, $G(t,{x})$ is the heat kernel defined as
\be \label{heat-ker}
G(t, {x}) := \dfrac{1}{(4\pi t)^{d/2}} e^{-|x|^2/4t}.
\ee

The analysis of the zero level set is important to understanding the behavior of sign-changing solutions of parabolic equations.
To analyze the behavior of sign-changing solutions, many researchers have focused  on the zero level set or the sign-changing number of the initial data \cite{Angenent, Chung, DGM, Matano, Mizoguchi}.
Especially, the case $d=1$ is particularly well analyzed.
In parabolic equations, one of the most fundamental properties is that $\cZ(t)$ is a discrete set \cite{Angenent}.
Furthermore, it is known that the number of elements of $\cZ(t)$
does not increase as time passes \cite{DGM, Matano}.

In investigating the dynamics of the sign-changing solution in detail,
the asymptotic behavior of $\cZ(t)$ has been analyzed to the heat equation with $d=1$.
Mizoguchi \cite{Mizoguchi} estimated the upper bound of $\cZ(t)$ in the case that $u_0(x)$ is sign-changing at finite times
and proved $\cZ(t)\subset [-C_0 t, C_0 t]$ for sufficiently large $t>0$ with some $C_0>0$.
Moreover, it has been reported that there exist $u_0(x)$ and $x^{*}(t)\in \cZ(t)$ such that
$x^{*}(t)>C_1 t$ holds for sufficiently large $t>0$ with some $C_1>0$.
Chung \cite{Chung} analyzed the asymptotic behavior of zero points when $u_0(x)$ belongs to $L^1(\bR, 1+ |x|^{k+1})$ for some $k\in\bN$ and satisfies
\be\label{Assu}
\int_{\bR} y^{j} u_0(y)dy =0\ (j=0,1,\ldots,k-1),\quad  \int_{\bR} y^{k} u_0(y)dy \neq 0.
\ee
As the result of \cite{Chung}, for sufficiently large $t>0$, there exist $x^{*}_{j}(t)\in \cZ(t)\ (j=1,2,\ldots,k)$
such that 
\beaa
\dlim_{t\to+\infty} \dfrac{x^{*}_{j}(t)}{2\sqrt{t}} =h_j,
\eeaa
where $h_j\ (j=1,2,\ldots,k)$ are mutually different zero points of the Hermite polynomial $H_k(x)$ with the order $k$ defined as 
\beaa
H_k(x):= (-1)^{k} e^{x^2} \dfrac{d^k}{dx^k}[e^{-x^2}].
\eeaa
The proof of these results is based on the Hermite polynomial approximation deduced using the asymptotic expansion for the heat equation (see \cite{Chung} and the references therein).
Chung further analyzed the behavior of the zero point of the solution in the case $d\ge 2$ and also obtained results for the spherically symmetric initial data.

We expect from these results that the elements of $\cZ(t)$ have various asymptotic profiles.
However, even the case $d=1$ does not yield a clear answer.
Especially, even though there is a case that the asymptotic behavior of an element of $\cZ(t)$ is $O(t)$ as $t\to+\infty$,
the characterization of its coefficient has not been given.
Indeed, it seems interesting that the asymptotic behavior of some zero point is $O(t)$ 
while $\sqrt{t}$ is the scale of self-similarity of the solution of the heat equation.
The aim of this paper is to explicitly give asymptotic profiles of zero points
and to reveal what properties of the initial data make a difference in the orders.

As related problems, the level set of the solution to the heat equation has been analyzed theoretically in connection with the hot spot problems \cite{CK, Ishige1, Ishige2, JS} 
and for application aspects such as level set methods \cite{MBO} and image processing \cite{GGOU, HIH}.
Among them, we mention the hot spot problem in the heat equation.
It is the study of the behavior of the maximum point of a solution, and the maximum points are called hot spots.
Under the non-negativity of the initial condition and relatively general conditions, we can show that the hot spots are critical points as pointed out in \cite{CK, JS}.
Since the partial derivatives of the solution are also solutions to the heat equation, 
analyzing the behavior of the zero point also allows analysis of hot spots through critical points.
In this sense, the hot spot problem is an issue related to the motivation of this paper.
The hot spot problem for the heat equation on unbounded domains in high dimensional space is often analyzed by imposing some non-negativity on the initial data \cite{CK, Ishige1, Ishige2, JS}.
To the best of our knowledge, hot spots asymptotically move slower than $O(t)$ in many cases, and there is no consideration of zero points moving at $O(t)$ which is the aim of this paper.

In this paper, we develop higher-order asymptotic expansions of the heat equation in order to analyze the behavior of the zero points in detail.
Furthermore, based on the analysis, 
the asymptotic behavior of the zero point in the case $d=1$ is revealed up to the constant term.
This result provides a complete answer for the asymptotic behavior of the zero point in the one-dimensional case 
when the initial data has fast decaying and finite sign changes.

The paper is organized as follows:
Section \ref{sec:pre} presents results on higher-order asymptotic expansions on moving coordinate systems for analyzing the asymptotic behavior of zero points.
In Section \ref{sec:d=1}, the asymptotic behavior of the zero point for $d=1$ is analyzed in detail with the proofs.
We summarize the analysis of specific cases for the higher dimensional case in Section \ref{sec:d>1}.
Finally, we mention related problems in Section \ref{sec:dis}.

\section{Preliminaries}\label{sec:pre}
We prepare symbols and notations to explain the analytical methods used in this study.
For all multi index $\alpha=(\alpha_1,\ldots,\alpha_d)\in (\bN\cup \{0\})^{d}$, we define the sum of components as  $|\alpha|_{d}:=\dsum^{d}_{j=1}\alpha_j$.
Also, the $\alpha$ partial derivative to the smooth function $f(x)$ is defined as follows:
\beaa
\partial^{\alpha}_{x}f(x):= \dfrac{\partial^{|\alpha|}f}{\partial^{\alpha_1}_{x_1}\cdots \partial^{\alpha_d}_{x_d}}(x).
\eeaa
We also define $x^{\alpha}=x^{\alpha_1}_1\cdots x^{\alpha_d}_d$ and $\alpha! = \alpha_1 !\cdots \alpha_d!$.

Under this preparation, it is well known that higher-order asymptotic expansions for solutions such as the following:
\begin{lemma}\cite[Theorem 2.2]{Chung}\label{Her-app}
Let $u$ be the bounded solution on $(0,+\infty)\times\bR^d$ to the heat equation \eqref{eq:main} with initial data $u_0\in L^1(\bR^d; 1+|x|^{n+1})$
for some nonnegative integer $n$.
There is then a positive constant $C=C(n,d)$ such that
\beaa
\left|u(t,x)- G(t,x) \dsum_{|\alpha|\le n} \dfrac{\int_{\bR^d} y^{\alpha} u_0(y)dy}{(4t)^{|\alpha|/2}  \alpha!} \prod^{d}_{j=1}  H_{\alpha_j}\left(\dfrac{x_j}{2\sqrt{t}} \right) \right| \le C t^{- (n+d+1)/2} \| |\cdot|^{n+1}u_0 \|_{L^1}
\eeaa
for all $x\in\bR^d$.
\end{lemma}

We analyze the asymptotic behavior of the zero point by considering a natural extension of this asymptotic expansion.
The following conditions are given for the initial data:
\be\label{condi1}
\forall \lambda\in\bR^{d},\quad \int_{\bR^{d}} e^{\lng \lambda,y\rng} |u_0(y)| dy<\infty. 
\ee
A concrete example satisfying \eqref{condi1} is bounded functions with compact support.

We assume that $u_0$ satisfies \eqref{condi1}.
For any $\eta\in\bR^{d}$, 
consider $v(t,x):= e^{<\eta,x>+t|\eta|^2}u(t,x+2t\eta)$.
Then, $v(t,x)$ is a solution of Cauchy problem
\beaa
\begin{cases}
\dfrac{\partial v}{\partial t}=\Delta v \quad (t>0,\ x\in\bR^{d}), \vspace{3mm}\\
v(0,x)= e^{<\eta, x>} u_0(x) \quad (x\in\bR^d).
\end{cases}
\eeaa
Since $v(0,x)$ belongs to $L^1(\bR^d; 1+|x|^{n+1})$ for all $n\ge 0$,
we can apply Lemma \ref{Her-app}.
We define the bilateral Laplace transform as
\beaa
U_0(\eta):= \int_{\bR^{d}} e^{\lng \eta , y \rng} u_0(y) dy,\quad \eta=(\eta_1,\ldots,\eta_d)\in\bR^{d}.
\eeaa
Notice that
\beaa
\partial^{\alpha}_{\eta}U_0(\eta)=\int_{\bR^d} y^{\alpha}e^{\lng \eta, y\rng} u_0(y)dy = \int_{\bR^d} y^{\alpha}v(0,y)dy,
\eeaa
we can deduce the following lemma.

\begin{lemma}\label{lem:high}
Let $u$ be the bounded solution on $(0,+\infty)\times\bR^d$ to the heat equation \eqref{eq:main} with initial data $u_0\in L^1(\bR^{d})\cap L^{\infty}(\bR^{d})$ satisfying \eqref{condi1}.
For any $\eta\in\bR$ and non-negative integer $n$,
there is a positive constant $C=C(n,d)$ such that $v(t,x)= e^{<\eta,x>+t|\eta|^2}u(t,x+2t\eta)$ satisfies
\beaa
\left|v(t,x)- G(t,x) \dsum_{|\alpha|\le n} \dfrac{\partial^{\alpha}_{\eta}U_0(\eta)}{(4t)^{|\alpha|/2}  \alpha!} \prod^{d}_{j=1}  H_{\alpha_j}\left(\dfrac{x_j}{2\sqrt{t}} \right) \right| \le C t^{- (n+d+1)/2} \| e^{\lng \eta,\cdot\rng}|\cdot|^{n+1}u_0 \|_{L^1}
\eeaa
for all $x\in\bR$.
\end{lemma}

Let us define the zero level set of $U_0$
\beaa
\cN(U_0):=\{ \eta\in\bR^{d} \mid U_0(\eta)=0 \}.
\eeaa
Assume that $\cN(U_0)\neq \emptyset$ and fix $\eta^{*}\in\cN(U_0)$.
Since $U_0(\eta)$ is a analytic function on $\bC^{d}$, 
there exists $k\in\bN$ such that
for all $\alpha=(\alpha_1,\ldots,\alpha_d)\in (\bN\cup \{0\})^{d}$ with $|\alpha|:=\dsum^{d}_{j=1}\alpha_j <k$,
we have $\partial^{\alpha}_{\eta} U_0 (\eta^{*})=0$,
and there exists $\tilde{\alpha}\in (\bN\cup \{0\})^{d}$ satisfying $|\tilde{\alpha}|=k$ and
$\partial^{\tilde{\alpha}}_{\eta} U_0 (\eta^{*}) \neq 0$.
In this case, $\eta^{*}$ is called a zero point of $U_0(\eta)$ with multiplicity $k$.

From Lemma \ref{lem:high}, if $\cN(U_0)\neq \emptyset$ and $\eta^{*}\in\cN(U_0)$ has multiplicity $k$, then there is a $C=C(k,d,\eta^{*},u_0)$ such that
\beaa
\left|v(t,x)- \dfrac{G(t,x)}{(4t)^{k/2} } \dsum_{|\alpha| = k} \dfrac{\partial^{\alpha}_{\eta}U_0(\eta^{*})}{ \alpha!} \prod^{d}_{j=1}  H_{\alpha_j}\left(\dfrac{x_j}{2\sqrt{t}} \right) \right| \le C t^{- (k+d+1)/2}
\eeaa
for any $x\in\bR^d$. 
This means that
\be\label{lim:unif}
(4t)^{(d+k)/2} v(t,2\sqrt{t} x) \to \dfrac{e^{-|x|^2}}{\pi^{d/2}} \dsum_{|\alpha|=k} \dfrac{\partial^{\alpha}_{\eta}U_0(\eta^{*})}{ \alpha!} \prod^{d}_{j=1}  H_{\alpha_j}\left(x \right) =: \tilde{H}_{d,k}(x;\eta^{*})
\ee
holds uniformly as $t\to+\infty$.
Since it is uniform convergence, we can state the existence of zero points depending on the properties of $\tilde{H}_{d,k}(x;\eta^{*})$.
We have not yet found a unified method for this issue, and thus the analysis must be done in a case-by-case manner.

\begin{remark}
We note that the inequality in Lemma \ref{lem:high} can be rewritten as
\beaa
\left|v(t,x)- \dsum^{n}_{m=1} \dfrac{1}{(4t)^{(d+m)/2}} \tilde{H}_{d,m}\left(\dfrac{x}{2\sqrt{t}};\eta^{*} \right) \right| \le C t^{- (n+d+1)/2} \| e^{\lng \eta,\cdot\rng}|\cdot|^{n+1}u_0 \|_{L^1}.
\eeaa
\end{remark}

\begin{remark}
The condition \eqref{condi1} can be discussed more weakly for the fast-decaying initial data.
However, it cannot be extended to the case of slow decay, as in polynomial decay.
\end{remark}

\section{The case $d=1$}\label{sec:d=1}
Here we analyze the case $d=1$ in detail.
We first consider the asymptotic behavior of a zero point that is $O(t)$
\begin{proposition}\label{prop:bound}
Suppose that $u_0\in L^1(\bR)\cap L^{\infty}(\bR)$ satisfies \eqref{condi1}.
Assume that there exist $T>0$ and $x^{*}(t)\in Z(t)$ for $t>T$ satisfying $x^{*}\in C(T,\infty)$ and
\be\label{limsup}
\limsup_{t\to +\infty} \left| \dfrac{x^{*}(t)}{2t} \right| <\infty.
\ee
Then, $\theta:= \dlim_{t\to+\infty} \dfrac{x^{*}(t)}{2t}$ exists and belongs to $\cN(U_0)$.
\end{proposition}
\begin{proof}
We first define
\beaa
\overline{\theta}:= \limsup_{t\to+\infty} \dfrac{x^{*}(t)}{2t},\quad \underline{\theta}:= \liminf_{t\to+\infty} \dfrac{x^{*}(t)}{2t}.
\eeaa
and then fix $\theta_0\in [\underline{\theta},\overline{\theta}]$.
Because $x^{*}(t)\in C(T,\infty)$, there exists $\{t_j\}_{j\in\bN}\subset (T,\infty)$ such that $t_j\to+\infty$ and
$x^{*}(t_j)/2t_{j}\to \theta_0$ hold as $j\to+\infty$.
From the assumption that $x^{*}(t)\in \cZ(t)$ for $t>T$,
we have
\beaa
u(t,x^{*}(t_j))=(G(t)*u_0)(x^{*}(t_j))=0 \quad \Leftrightarrow \quad \int_{\bR} \exp\left( \frac{2x^{*}(t_j) y - y^2}{4t_j} \right)u_0(y) dy = 0
\eeaa
for any $j\in\bN$.
From the dominated convergence theorem, 
we obtain 
\beaa
0= \int_{\bR} e^{\theta_0 y} u_0(y) dy = U_0(\theta_0)
\eeaa
as $j\to\infty$.
This means that $\theta_0=\overline{\theta}=\underline{\theta}$ and $\theta_0\in \cN(U_0)$, because $U_0$ is an analytic function on $\bC$ and thus $\cN(U_0)$ is discrete.
\end{proof}
From this result, the asymptotic profiles corresponding to $O(t)$ are expected to move at a certain speed without oscillating.
In particular, if the upper bound of $\cZ(t)$ is $O(t)$, the asymptotic behavior of $\cZ(t)$ is expected to change depending on the nature of $\cN(U_0)$.

We next mention the case that $\cN(U_0)=\emptyset$.
Before describing the result, we introduce a notation.
For a function $f$ on $\bR$ with $f(x)\not\equiv 0$, let $z(f)$ be the number of sign changes; i.e. the supremum of $j$ such that
\beaa
f(x_{i})f(x_{i+1})<0,\quad i=1,2,\ldots,j
\eeaa
for some $-\infty<x_1<x_2<\ldots<x_{j+1}<+\infty$.

If $f(x)$ satisfies $z(f)<+\infty$,
then \eqref{limsup} in the statement of Proposition \ref{prop:bound} always holds by Theorem 1.1 in \cite{Mizoguchi}.
We thus deduce the following result from a proof by contradiction.

\begin{corollary}\label{cor:emp}
Suppose that $u_0\in L^1(\bR)\cap L^{\infty}(\bR)$ satisfies \eqref{condi1}.
Assume that $\cN(U_0)= \emptyset$ and $z(u_0)<\infty$.
Then, there exists $T>0$ such that $\cZ(t)=\emptyset $ for any $t>T$.
\end{corollary}

We analyze the case $\cN(U_0)\neq \emptyset$.
By applying extended asymptotic expansion, we obtain the following result:
\begin{theorem}\label{thm:mainA}
Suppose that $u_0\in L^1(\bR)\cap L^{\infty}(\bR)$ satisfies \eqref{condi1}.
Assume that $\cN(U_0)\neq \emptyset$.
Then, for all $\eta^{*}\in\cN(U_0)$ with multiplicity $k\in\bN$,
there exist $T>0$ and $x^{*}_j(t)\in \cZ(t)\ (j=1,\ldots,k)$ for $t>T$ satisfying
\beaa
x^{*}_j(t)= 2t \eta^{*}+ 2\sqrt{t}h_j + \dfrac{U^{(k+1)}_0(\eta^{*})}{(k+1) U^{(k)}_0(\eta^{*})}+o(1) \quad (t\to +\infty)
\eeaa
for $j=1,\ldots,k$, where $h_j\ (j=1,\ldots,k)$ are mutually different zero points of the Hermite polynomial $H_k(x)$ 
with the order $k$.
\end{theorem}
\begin{proof}
Let $v(t,x)= e^{\eta^{*} x+t(\eta^{*})^2}u(t,x+2t\eta^{*})$,
and analyze zero points of $v(t,x)$ 
By applying Lemma \ref{lem:high} with $d=1$, $\eta=\eta^{*}$ and $n=k+1$,
there is $C>0$ such that
\beaa
\left|v(t,x)- G(t,x) \dsum^{k+1}_{m=k} \dfrac{U^{(m)}_0(\eta^{*})}{(4t)^{m/2} m!} H_{m}\left(\dfrac{x}{2\sqrt{t}} \right) \right| \le C t^{- (k+3)/2}
\eeaa
for all $x\in\bR$.
Let $h$ be a zero point of $H_k(x)$.
Then, we rewrite the inequality as 
\beaa
\left| v(t, 2\sqrt{t}h + x)- G(t,2\sqrt{t}h + x) \dsum^{k+1}_{m=k} \dfrac{U^{(m)}_0(\eta^{*})}{ (4t)^{m/2} m !} H_{m}\left( h + \dfrac{x}{2\sqrt{t}} \right) \right| \le C t^{- (k+3)/2}
\eeaa
by simple computation and changing variable $x$ as $2\sqrt{t} h + x$.
Note that $C$ is re-taken here.
Since we know that $H'_{k}(h)=2hH_{k}(h)-H_{k+1}(h)=-H_{k+1}(h)\neq 0$, 
we deduce
\beaa
&&(4t)^{(k+2)/2}G(t,2\sqrt{t}h + x) \dsum^{k+1}_{m=k} \dfrac{U^{(m)}_0(\eta^{*})}{ (4t)^{m/2} m !} H_{m}\left( h + \dfrac{x}{2\sqrt{t}} \right) \\
&& = \exp\left(-\dfrac{ ( x+ 2\sqrt{t}h)^2 }{4t} \right)
\left\{ 2\sqrt{t} \dfrac{U^{(k)}_0(\eta^{*})}{k!}  H_{k}\left( h + \dfrac{x}{2\sqrt{t}} \right) \right. \\
&& \hspace{6cm}\left.+ \dfrac{U^{(k+1)}_0(\eta^{*})}{(k+1)!}  H_{k+1}\left( h + \dfrac{x}{2\sqrt{t}} \right) \right\} \\
&& \to e^{-h^2} \left\{ \dfrac{U^{(k)}_0(\eta^{*})}{k!} H'_k(h) x  + \dfrac{U^{(k+1)}_0(\eta^{*})}{(k+1)!} H_{k+1}(h) \right\} \quad (t\to+\infty) \\
&&= \dfrac{e^{-h^2} H_{k+1}(h) }{k!} \left\{ \dfrac{U^{(k+1)}_{0}(\eta^{*})}{k+1} -U^{(k)}_0(\eta^{*})x\right\}.
\eeaa
We note that the limit is right in the sense of local uniforms.
We also know the convergence rate of the limit, which allows us to construct a zero point by using the intermediate value theorem.
Thus, there exist $T>0$ and $p^{*}(t)\in\bR$ for $t>T$ such that
\beaa
v(t,p^{*}(t))= 0, \quad 
p^{*} (t) = 2\sqrt{t} h + \dfrac{U^{(k+1)}_{0}(\eta^{*})}{(k+1)U^{(k)}_{0}(\eta^{*})} + o(1)\quad (t\to+\infty).
\eeaa
Thus, there exist such zero point $p^{*}_j(t)$ of $v(t,x)$ for each $h_j$,.

From the above discussion, setting $x^{*}_j(t)= 2t \eta^{*} + p^{*}_j(t)\ (j=1,2,\ldots,k)$ yields $x^{*}_j (t)\in \cZ(t)$ and the desired asymptotic behavior.
\end{proof}

\begin{remark}
We can compose $x^{*}_j(t)\in \cZ(t)\ (j=1,\ldots,k)$ obtained in Theorem \ref{thm:mainA} as $x^{*}_{j}\in C(T,\infty)\ (j=1,\ldots,k)$ by applying the argument and results of \cite{Angenent}. 
\end{remark}
From Theorem \ref{thm:mainA}, for any $\eta^{*}\in \cN(U_0)$ and sufficiently large $t>0$, 
there exists $x^{*}(t)\in \cZ(t)$ such that 
\beaa
\dlim_{t\to+\infty} \dfrac{x^{*}(t)}{2t} = \eta^{*}.
\eeaa
Thus, if $\eta^{*}\neq 0$, then $x^{*}(t)$ diverges with $O(t)$ as time passes.
In the case that $0\in \cN(U_0)$ with multiplicity $k>1$, 
for sufficiently large $t>0$, there exists $x^{*}(t)\in \cZ(t)$ such that 
\beaa
\dlim_{t\to+\infty} \dfrac{x^{*}(t)}{2\sqrt{t}} = h,
\eeaa
where $h$ is a zero point of $H_k(x)$.
When $0\in \cN(U_0)$ with multiplicity $k$ and $k$ is odd,
for sufficiently large $t>0$, there is $x^{*}(t)\in Z(t)$ satisfying 
\beaa
\dlim_{t\to+\infty} x^{*}(t) =  \dfrac{U^{(k+1)}_0(0)}{(k+1) U^{(k)}_0(0)} 
= \dfrac{\dint_{\bR} y^{k+1}u_0(y)dy}{(k+1)\dint_{\bR} y^{k} u_0(y)dy}.
\eeaa
Hence, by analyzing the zero points of $U_0(\eta)$ and their multiplicity,
we can understand the long-time behavior of elements of $\cZ(t)$.

\section{The case $k=1$ in high dimensional space}\label{sec:d>1}
The high-dimensional case is discussed.
In this case, no unified results have yet been obtained, 
and even the upper bound of the zero level set has not been clarified.
For the initial data $u_0(x)$ that is radial symmetric,
Chung \cite{Chung} considered the case that $\eta^{*}=0\in\cN(U_0)$ and proved that 
$\tilde{H}_{d,k}(x;\eta^{*})$ can be represented using the generalized Laguerre polynomial.
Chung also analyzed the case that $\eta^{*}=0\in\cN(U_0)$ with mulitiplicity $k=1,2$ for $d=2$.
This reveals that there exists a zero point that is $O(\sqrt{t})$ even in higher dimensions.
However, when $\tilde{\eta}\neq 0$, 
the asymptotic behavior of the zero point of the solution is not obtained.
In particular, the asymptotic profile of $O(t)$ is not known.

The analysis under general assumptions is a future problem, and here we summarize the result of asymptotic profiles of zero points for the case $d\ge 2$, $\cN(U_0)\neq \emptyset$, and $k=1$, i.e. zero points of $U_0(\eta)$ have multiplicity 1.
Let $\eta^{*}\in \cN(U_0)$ with multiplicity 1.
Then, we obtain
\beaa
\tilde{H}_{d,1}(x,;\eta^{*}) = \dfrac{2e^{-|x|^2}}{\pi^{d/2}} \dsum^{d}_{j= 1} \dfrac{\partial U_0}{\partial \eta_j}(\eta^{*}) x_{j} =  \dfrac{2e^{-|x|^2}}{\pi^{d/2}} \lng \nabla_{\eta} U_0(\eta^{*}), x\rng
\eeaa
by using $H_0(p)=1,\ H_1(p)=2p$ for $p\in\bR$.
Then, the following theorem is obtained as a result which is a natural extension of the case $d=1$.
\begin{theorem}
Suppose that $u_0\in L^1(\bR^d)\cap L^{\infty}(\bR^d)$ satisfies \eqref{condi1}.
Assume that $\cN(U_0)\neq \emptyset$.
Then, for all $\eta^{*}\in\cN(U_0)$ with multiplicity $1$,
there exist $T>0$ and $x^{*}(t)\in \cZ(t)$ for $t>T$ satisfying
\beaa
x^{*}(t)= 2t \eta^{*}+ \dfrac{\Delta_{\eta} U_0(\eta^{*})}{2|\nabla_{\eta} U_0(\eta^{*})|^2} \nabla_{\eta} U_0(\eta^{*}) +o(1) \quad (t\to +\infty).
\eeaa
\end{theorem}
\begin{proof}
By applying Lemma \ref{lem:high} with $n=2$ and multiplying $(4t)^{(d+2)/2}$ by both sides of the inequality, we have
\beaa
\left| (4t)^{(d+2)/2}v(t, x)- \left\{ 2\sqrt{t} \tilde{H}_{d,1}\left( \dfrac{x}{2\sqrt{t}} ; \eta^{*}\right) +  \tilde{H}_{d,2}\left(\dfrac{x}{2\sqrt{t}}\right) \right\}  \right|  \le C t^{- 1/2} 
\eeaa
for all $x\in\bR^{d}$ with a positive constant $C>0$.
Hence, for all $x\in\bR$, 
\be\label{eq:conv}
(4t)^{(d+2)/2}v(t,x) \to \lng \nabla_{x}\tilde{H}_{d,1}\left( 0 ; \eta^{*}\right), x \rng +  H_{d,2}\left(0;\eta^{*}\right) 
\ee
holds as $t\to+\infty$ in the sense of uniform convergence.
Since we deduce
\beaa
\dfrac{\partial \tilde{H}_{d,1}}{\partial x_m} (x;\eta) = - \dfrac{4 x_m e^{-|x|^2} }{\pi^{d/2}} \dsum^{d}_{j= 1} \dfrac{\partial U_0(\eta^{*})}{\partial \eta_j} x_{j} + \dfrac{2e^{-|x|^2}}{\pi^{d/2}}  \dfrac{\partial U_0(\eta^{*})}{\partial \eta_m}
\eeaa
for $m=1,2,\ldots,d$, we get $\nabla_{x}\tilde{H}_{d,1}\left( 0 ; \eta^{*}\right) = \dfrac{2}{\pi^{d/2}}\nabla_{\eta} U_0 (\eta^{*})$.
Also, we have
\beaa
H_{d,2}(0) &=& \dfrac{1}{\pi^{d/2}} \dsum_{|\alpha|=2} \dfrac{\partial^{\alpha}U_0(\eta^{*})}{ \alpha!} \prod^{d}_{j=1}  H_{\alpha_j}\left(0 \right) \\
&=& \dfrac{1}{2 \pi^{d/2}} \dsum^{d}_{j=1} \dfrac{\partial^2 U_0(\eta^{*})}{ \partial \eta^2_j}  H_{2}\left(0 \right) \\
&=& - \dfrac{1}{\pi^{d/2}} \Delta_{\eta} U_0 (\eta^{*}). 
\eeaa
Hence, the convergence of \eqref{eq:conv} is rewritten as
\beaa
\lng \nabla_{x}\tilde{H}_{d,1}\left( 0 ; \eta^{*}\right), x \rng +  H_{d,2}\left(0;\eta^{*}\right) = \dfrac{1}{\pi^{d/2}} \{ 2 \lng \nabla_{\eta} U_0(\eta^{*}), x\rng - \Delta_{\eta} U_0(\eta^{*})  \}.
\eeaa
Then, $x=\dfrac{\Delta_{\eta} U_0(\eta^{*})}{2|\nabla_{\eta} U_0(\eta^{*})|^2} \nabla_{\eta} U_0(\eta^{*}) \in\bR^{d}$ satisfies
\beaa
2 \lng \nabla_{\eta} U_0(\eta^{*}), x\rng - \Delta_{\eta} U_0(\eta^{*}) = 0.
\eeaa
Since \eqref{eq:conv} is uniformly convergent and the convergence rate is known, we can construct the zero point by using the intermediate value theorem and the implicit function theorem.
Therefore, there exist $T>0$ and $p^{*}(t)\in\cZ(t)$ for $t>T$ satisfying
\beaa
v(t,p^{*}(t))= 0 \ (t>T),\quad p^{*}(t)=  \dfrac{\Delta_{\eta} U_0(\eta^{*})}{2|\nabla_{\eta} U_0(\eta^{*})|^2} \nabla_{\eta} U_0(\eta^{*}) + o(1)\quad (t\to+\infty).
\eeaa

Thus, setting $x^{*}(t)= 2t \eta^{*} + p^{*}(t)$ yields $x^{*}(t)\in \cZ(t)$ and the desired asymptotic behavior.
\end{proof}

Here we treat the case of the radial symmetric initial data as a corollary.
Let $u_0$ be radial symmetric function on $\bR^{d}$ satisfying \eqref{condi1}.
Then, $U_0(\eta)$ is also a radial symmetric function on $\bR^{d}$.
We assume that $\cN(U_0)\neq\emptyset$.
Since $U_0(\eta)$ is an analytic function, $\cN(U_0)$ must be discrete with respect to the radial direction.
Then, $U_0(\eta) = U_s(|\eta|)$ satisfies
\beaa
\dfrac{\partial U_0}{\partial \eta_j}(\eta) = \dfrac{\eta_j}{|\eta|} U'_s(|\eta|), \quad 
\dfrac{\partial^2 U_0}{\partial \eta^2_j}(\eta) = \dfrac{\eta^2_j}{|\eta|^2} U''_s(|\eta|) + \left(\dfrac{1}{|\eta|} - \dfrac{\eta^2_j}{|\eta|^{3/2}} \right) 
 U'_s(|\eta|)
\eeaa
for $|\eta|>0$ and $j=1,2,\ldots,d$.
Hence, 
\beaa
\dfrac{\Delta_{\eta} U_0(\eta)}{2|\nabla_{\eta} U_0(\eta)|^2} \nabla_{\eta} U_0(\eta) = \dfrac{1}{2 }  \left( \dfrac{U''_s(|\eta|)}{U'_s(|\eta|)} + \dfrac{d-1}{|\eta|} \right) \dfrac{\eta}{|\eta|}
\eeaa
holds if $U'_s(|\eta|)\neq 0$.
This yields the following result:
\begin{corollary}
Suppose that $u_0\in L^1(\bR^d)\cap L^{\infty}(\bR^d)$ is a radial function satisfying \eqref{condi1}.
Assume that $\cN(U_0)\neq \emptyset$.
Then, for any $\eta^{*}\in\cN(U_0)$ with multiplicity $1$, there exist $T>0$ and $r^{*}(t)>0$ for $t>T$ satisfying
\beaa
\{ x\in\bR^{d} \mid |x|= r^{*}(t) \} \subset \cZ(t) 
\eeaa
for $t>T$ and 
\beaa
r^{*}(t) = 2t |\eta^{*}| + \dfrac{1}{2 }  \left( \dfrac{U''_s(|\eta^{*}|)}{U'_s(|\eta^{*}|)} + \dfrac{d-1}{|\eta^{*}|} \right) + o(1) \quad (t\to+\infty).
\eeaa
\end{corollary}
\begin{remark}
Note that $k>1$ when $\eta^{*}=0\in\cN(U_0)$. 
In such a case, Chung \cite{Chung} has obtained a result on the behavior of zero points moving in $O(\sqrt{t})$.
\end{remark}
When $k\ge 2$, we may see the behavior of term corresponding to $O(\sqrt{t})$ depending on the nature of $\tilde{H}_{d,k}(x;\eta^{*})$, but it is difficult to find general properties about the zero point. 
There may be some good ways to analyze them, but we have not found them yet.
This point is currently left as an open problem.

\section{Related problems}\label{sec:dis}
Finally, some related problems are mentioned as developmental questions.
\subsection{Hot spot problems}
In Euclidean space $\bR^{d}$, it is known that when the initial data is non-zero and non-negative, the hot spot converges to a point determined by the initial data.
When the initial data is sign-changing, the asymptotic behavior of the hot spot is different from the case that the initial data is non-negative. 
In fact, we can construct an example where all diverge at the zero point of $u_{x}(t,x)$, which is a solution of the heat equation in the case of $d=1$.
This means that it corresponds to the case where all the critical points, which include the maximum and minimum points of the solution of the heat equation, diverge.
To reveal the motion of hot spots, a more detailed evaluation is needed to determine which critical points will become hot spots. 
The hot spot problem for the initial data that is sign-changing is one interesting future problems.

\subsection{Reaction-diffusion equations}
It is also interesting to investigate the dynamics near the constant state of the reaction-diffusion equation:
\beaa
\dfrac{\partial u}{\partial t} =\Delta_{x} u + f(u) \quad (t>0,\ x\in\bR^d),
\eeaa
where $u=u(t,x)\in\bR$ and $f:\bR\to\bR$ is a smooth function.
If $f(u)$ satisfies $f(0)=0$,
then we can expect that the solution dynamics can be approximated by a solution of a linearized equation around $u=0$ as long as the solution is sufficiently small.
Hence, the behavior of the zero point of the solution can be analyzed by attributing it to the diffusion equation, and the influence of the nonlinear term may be seen.

One particularly interesting topics is the case of Fujita-type nonlinear terms $f(u)= |u|^{p-1}u$ for $p>1$.
Mizoguchi \cite{Mizoguchi} discussed the upper bound of the zero level set when $d=1$.
It reported that the zero level set has similar properties to the diffusion equation when $p>3$ and the decay of the global solution is sufficiently fast.
Furthermore, Mizoguchi and Yanagida \cite{MY} reported that a global solution exists for $d=1$ by changing the sign of the initial data even when $p$ is smaller than the Fujita exponent.
They proposed that a sign change weakens the condition for the global existence of the solution to the Fujita equation.
The analysis of the zero level set may therefore be an indicator for considering the conditions for the global existence of a solution.
Analysis using higher-order asymptotic expansions of solutions such as this paper may be useful and should be considered in the future.

\subsection{Nonlocal diffusion equations}
The asymptotic behavior of the zero-level set of another diffusion equation can be considered as another extension of the problem.
As an example, we consider the case of the space fractional diffusion equation:
\beaa
\begin{cases}
\dfrac{\partial u}{\partial t}+(-\Delta_{x})^{s}u =0 \quad (t>0,\ x\in\bR^d), \vspace{3mm}\\
u(0,x)=u_0(x) \quad (x\in\bR^d),
\end{cases}
\eeaa
where $(-\Delta)^{s}$ is the fractional power of the Laplace operator with $0<s<1$
In this case, it can be deduced that the zeros disappear at a finite time when $u_0\in L^1(\bR^d)$ is compactly supported and its integral value is non-zero, 
because the evaluation of the relative error has been derived \cite{Luis}. 
In addition, when the integral value of the initial data is zero, the asymptotic behavior of the zero level set can be considered.
The specific behavior of the zero level set can be considered using the asymptotic expansion of the solution \cite{IKM}.

Furthermore, the asymptotic behavior of the zero point of the solutions to the time-fractional diffusion equation (the diffusion-wave equation) \cite{EK} and 
the nonlocal diffusion equation \cite{AMRT} may also be considered, but we leave them as open problems.

\section*{Acknowledgments}
The author expresses his sincere gratitude to Prof. Shin-Ichiro Ei (Hokkaido University),
Prof. Eiji Yanagida (Tokyo Institute of Technology) and Prof. Kazuhiro Ishige (The University of Tokyo)  for their useful suggestions and valuable advice and thanks to Edanz (https://jp.edanz.com/ac) for editing a draft of this manuscript.
{The author is partially supported by JSPS KAKENHI Grant Number JP21J10036 and JP23K13013.}


\end{document}